\documentclass[reqno]{amsart}

\usepackage{amssymb}
\usepackage{graphicx}
\usepackage{amscd}
\usepackage[hidelinks]{hyperref}
\usepackage{color}
\usepackage{float}
\usepackage{graphics,amsmath,amssymb}
\usepackage{amsthm}
\usepackage{amsfonts}
\usepackage{latexsym}
\usepackage{epsf}
\usepackage{xifthen}
\usepackage{mathrsfs}
\usepackage{dsfont}
\usepackage{makecell}
\usepackage[FIGTOPCAP]{subfigure}
\usepackage{amsmath}
\allowdisplaybreaks[4]
\usepackage{listings}
\usepackage{etoolbox}
\usepackage{fancyhdr}
\usepackage{pdflscape}
\usepackage{afterpage}
\usepackage[title,toc,titletoc]{appendix}
\usepackage{enumitem}
\usepackage[noadjust]{cite}
\usepackage{tkz-graph}
\usepackage{tikz}
\usepackage{ulem}
\usepackage{forest}

 \newtheoremstyle{mytheorem}
 {3pt}
 {3pt}
 {\slshape}
 {}
 {\bfseries}
 {.}
 { }
 {}

\numberwithin{equation}{section}

\theoremstyle{theorem}
\newtheorem{theorem}{Theorem}[section]
\newtheorem*{theorem*}{Theorem}
\newtheorem{corollary}[theorem]{Corollary}

\newtheorem{proposition}[theorem]{Proposition}

\theoremstyle{definition}
\newtheorem{definition}{Definition}[section]

\newtheorem{example}{Example}[section]
\newtheorem*{example*}{Example}
\newtheorem{conjecture}{Conjecture}[section]

\theoremstyle{remark}
\newtheorem{remark}{Remark}[section]
\newtheorem*{remark*}{Remark}
\newtheorem*{remarks*}{Remarks}
\newtheorem{problem}{Problem}[section]

\newcommand{\Keywords}[1]{\ifthenelse{\isempty{#1}}{}{\smallskip \smallskip \noindent \textbf{Keywords}. #1}}
\newcommand{\MSC}[2][2010]{\ifthenelse{\isempty{#2}}{}{\smallskip \smallskip \noindent \textbf{#1MSC}. #2}}
\newcommand{\abstractnote}[1]{\ifthenelse{\isempty{#1}}{}{\smallskip \smallskip \noindent \textsuperscript{\dag}#1}}

\makeatletter
\def\specialsection{\@startsection{section}{1}%
  \z@{\linespacing\@plus\linespacing}{.5\linespacing}%
  {\normalfont}}
\def\section{\@startsection{section}{1}%
  \z@{.7\linespacing\@plus\linespacing}{.5\linespacing}%
  {\normalfont\scshape}}
\patchcmd{\@settitle}{\uppercasenonmath\@title}{\Large\boldmath}{}{}
\patchcmd{\@settitle}{\begin{center}}{\begin{flushleft}}{}{}
\patchcmd{\@settitle}{\end{center}}{\end{flushleft}}{}{}
\patchcmd{\@setauthors}{\MakeUppercase}{\normalsize}{}{}
\patchcmd{\@setauthors}{\centering}{\raggedright}{}{}
\patchcmd{\section}{\scshape}{\large\bfseries\boldmath}{}{}
\patchcmd{\subsection}{\bfseries}{\bfseries\boldmath}{}{}
\renewcommand{\@secnumfont}{\bfseries}
\patchcmd{\@startsection}{\@afterindenttrue}{\@afterindentfalse}{}{}
\patchcmd{\abstract}{\leftmargin3pc}{\leftmargin1pc}{}{}

\def\maketitle{\par
  \@topnum\z@ 
  \@setcopyright
  \thispagestyle{empty}
  \ifx\@empty\shortauthors \let\shortauthors\shorttitle
  \else \andify\shortauthors
  \fi
  \@maketitle@hook
  \begingroup
  \@maketitle
  \toks@\@xp{\shortauthors}\@temptokena\@xp{\shorttitle}%
  \toks4{\def\\{ \ignorespaces}}
  \edef\@tempa{%
    \@nx\markboth{\the\toks4
      \@nx\MakeUppercase{\the\toks@}}{\the\@temptokena}}%
  \@tempa
  \endgroup
  \c@footnote\z@
  \@cleartopmattertags
}
\makeatother

\newcommand{\mA}{\mathscr{A}}

\newcommand{\fF}{\mathfrak{F}}

\newcommand{\mG}{\mathscr{G}}

\newcommand{\mI}{\mathscr{I}}

\newcommand{\cL}{\mathcal{L}}

\newcommand{\mM}{\mathscr{M}}

\newcommand{\mP}{\mathscr{P}}

\newcommand{\mU}{\mathscr{U}}

\newcommand{\mV}{\mathscr{V}}

\newcommand{\mW}{\mathscr{W}}

\newcommand{\ubA}{\underline{\mathbf{A}}}

\newcommand{\ubF}{\underline{\mathbf{F}}}

\newcommand{\Mat}{\operatorname{Mat}}

\newcommand{\uba}{\underline{\boldsymbol{\alpha}}}
\newcommand{\ubb}{\underline{\boldsymbol{\beta}}}
\newcommand{\ubg}{\underline{\boldsymbol{\gamma}}}

\newcommand{\diag}{\operatorname{diag}}

\title{Linked partition ideals, directed graphs and $q$-multi-summations}

\author[S. Chern]{Shane Chern}
\address{Department of Mathematics, Penn State University, University Park, PA 16802, USA}
\email{shanechern@psu.edu}

\date{}

\begin{document}

%

\maketitle

\begin{abstract}

Finding an Andrews--Gordon type generating function identity for a linked partition ideal is difficult in most cases. In this paper, we will handle this problem in the setting of graph theory. With the generating function of directed graphs with an ``empty'' vertex, we then turn our attention to a $q$-difference system. This $q$-difference system eventually yields a factorization problem of a special type of column functional vectors involving $q$-multi-summations. Finally, using a recurrence relation satisfied by certain $q$-multi-summations, we are able to provide non-computer-assisted proofs of some Andrews--Gordon type generating function identities. These proofs also have an interesting connection with binary trees.

\Keywords{Linked partition ideal, directed graph, $q$-multi-summation, $q$-difference system, generating function, Andrews--Gordon type series.}

\MSC{Primary 11P84; Secondary 05A17, 05C05, 05C20, 33D70.}
\end{abstract}

\section{Introduction}

\subsection{Rogers--Ramanujan type identities}

The two Rogers--Ramanujan identities \cite{Ram1919,Rog1894}, which state as follows, have attracted a great deal of research interest in the theory of partitions.

\begin{theorem*}[Rogers--Ramanujan identities]
	\textup{(i).~}The number of partitions of a non-negative integer $n$ into parts congruent to $\pm 1$ modulo $5$ is the same as the number of partitions of $n$ such that each two consecutive parts have difference at least $2$.
	
	\textup{(ii).~}The number of partitions of a non-negative integer $n$ into parts congruent to $\pm 2$ modulo $5$ is the same as the number of partitions of $n$ such that each two consecutive parts have difference at least $2$ and such that the smallest part is at least $2$.
\end{theorem*}

There are many identities of the same flavor, including the Andrews--Gordon identity \cite{And1966,Gor1961}, the G\"{o}llnitz--Gordon identities \cite{Gol1967,Gor1966}, the Capparelli identities \cite{Cap1993} and so forth. In 2014, Kanade and Russell \cite{KR2015} further proposed six challenging conjectures on Rogers--Ramanujan type identities, the latter two of which were proved in 2018 by Bringmann, Jennings-Shaffer and Mahlburg \cite{BJM2018}.

Among these Rogers--Ramanujan type identities, two types of partition sets are considered. One partition set is consist of partitions under certain \textit{congruence} condition. For example, in the first Rogers--Ramanujan identity, we enumerate partitions into parts congruent to $\pm 1$ modulo $5$. The other partition set contains partitions under certain  \textit{difference-at-a-distance} theme. Let us first adopt a definition in \cite{KR2015}.

\begin{definition}
	We say that a partition $\lambda=\lambda_1+ \lambda_2+\cdots+ \lambda_{\ell}$ satisfies the \textit{difference at least $d$ at distance $k$} condition if, for all $j$, $\lambda_j - \lambda_{j+k}\ge  d$.
\end{definition}

\noindent In this setting, we may paraphrase the corresponding partition set in the first Rogers--Ramanujan identity as the set of partitions with difference at least $2$ at distance $1$.

\medskip

Although it is straightforward to find the generating function for partitions under given congruence condition, it is always difficult to obtain an analytic form of generating function for partitions under a difference-at-a-distance theme --- this is why the six conjectures of Kanade and Russell remained mysterious for years. But this problem was recently settled by Kanade and Russell themselves \cite{KR2018} and independently by Kur\c{s}ung\"{o}z \cite{Kur2018,Kur2019} using combinatorial approaches, and later by Li and the author \cite{CL2018} using algebraic methods. For example, in the Kanade--Russell conjecture $I_1$, we would like to count
\begin{quote}
	``partitions with difference at least $3$ at distance $2$ such that if two consecutive parts differ by at most $1$, then their sum is divisible by $3$.''
\end{quote}
It was shown that its generating function is a double summation as follows:
\begin{equation}\label{eq:I_1}
	\sum_{\lambda} x^{\sharp(\lambda)}q^{|\lambda|}=\sum_{n_1,n_2\ge 0}\frac{q^{n_1^2+3n_2^2+3n_1n_2}x^{n_1+2n_2}}{(q;q)_{n_1} (q^3;q^3)_{n_2}},
\end{equation}
where $\lambda$ runs through all such partitions, $\sharp(\lambda)$ denotes the number of parts in $\lambda$ and $|\lambda|$ is the size of $\lambda$ (that is, the sum of all parts in $\lambda$).

\subsection{Span one linked partition ideals}

In the 1970s, George Andrews \cite{And1972,And1974,And1975} have already started a systematic study of Rogers--Ramanujan type identities and developed a general theory in which the concept of \textit{linked partition ideals} was introduced. However, in this paper, we will not go into details of this concept due to its lengthy definition. The interested readers may refer to Chapter 8 of Andrews' book: \textit{The theory of partitions} \cite{And1976}.

What we are interested in this paper is a special case of linked partition ideals --- the \textit{span one linked partition ideals}. In fact, this special case is enough to cover most partition sets under difference-at-a-distance themes.

\medskip

Let us first fix some notations.

Let $\mathscr{P}$ be the set of all partitions. We define a map $\phi:\mP\to\mP$ by sending a partition $\lambda$ to another partition which is obtained by adding $1$ to each part of $\lambda$. For example, $\phi(5+3+3+2+1)=6+4+4+3+2$. For $k\ge 2$, we iteratively write $\phi^k(\lambda)=\phi(\phi^{k-1}(\lambda))$. Also, for two partitions $\lambda$ and $\pi$, their sum $\lambda\oplus\pi$ is constructed by counting the total appearances of each different part in $\lambda$ and $\pi$. For example, if $\lambda=3+2+1+1$ and $\pi=4+2+2+1+1$, then $\lambda\oplus\pi=4+3+2+2+2+1+1+1+1$.

Let $\Pi$ be a finite set of partitions containing the empty partition $\emptyset$. For each partition $\pi\in\Pi$, we define its \textit{linking set} $\mathcal{L}(\pi)$ by a subset of $\Pi$ containing the empty partition. Also, we require that the linking set of the empty partition, $\cL(\emptyset)$, equals $\Pi$. It is possible to construct finite chains
\begin{align}
\lambda_0\to\lambda_1\to\lambda_2\to\cdots\to\lambda_K
\end{align}
such that $\lambda_0\in\Pi$, $\lambda_K\ne \emptyset$ and for all $1\le k\le K$, $\lambda_k\in\cL(\lambda_{k-1})$. We may further extend such a finite chain to an infinite chain ending with a series of empty partitions
\begin{align}
\mathcal{C}: \lambda_0\to\lambda_1\to\lambda_2\to\cdots\to\lambda_K\to\emptyset\to\emptyset\to\cdots.
\end{align}
Let $S$ be a positive integer no smaller than the largest part among all partitions in $\Pi$. The above infinite chain $\mathcal{C}$ uniquely determines a partition by
\begin{align}\label{eq:decomp}
\lambda_0\oplus\phi^S(\lambda_1)\oplus\phi^{2S}(\lambda_2)\oplus\cdots\oplus \phi^{KS}(\lambda_K)\oplus \phi^{(K+1)S}(\emptyset)\oplus \phi^{(K+2)S}(\emptyset)\oplus\cdots,
\end{align}
which is equivalent to
\begin{align}
\lambda_0\oplus\phi^S(\lambda_1)\oplus\phi^{2S}(\lambda_2)\oplus\cdots\oplus \phi^{KS}(\lambda_K).
\end{align}
Let us collect such partitions along with the empty partition $\lambda=\emptyset$ (which corresponds to the infinite chain $\emptyset\to\emptyset\to\cdots$) and obtain a partition set $\mathscr{I}:=\mI(\langle\Pi,\cL\rangle,S)$. Then $\mI$ is called a \textit{span one linked partition ideal}.

\begin{example}\label{ex:1.1}
	In the first Rogers--Ramanujan identity, we consider partitions with difference at least $2$ at distance $1$. It is not hard to verify that this partition set is a span one linked partition ideal $\mI(\langle\Pi,\cL\rangle,S)$ where $\Pi=\{\emptyset,1,2\}$,\footnote{Here $1$ denotes a partition containing one part of size $1$ and likewise $2$ denotes a partition containing one part of size $2$.} the linking sets are
	$$\cL(\emptyset)=\{\emptyset,1,2\},\quad \cL(1)=\{\emptyset,1,2\},\quad \cL(2)=\{\emptyset,2\},$$
	and $S=2$.
\end{example}

\subsection{Generating function of span one linked partition ideals}

Given a span one linked partition ideal $\mI=\mI(\langle\Pi,\cL\rangle,S)$, one crucial problem is to determine its generating function
$$\mG(x)=\mG(x,q):=\sum_{\lambda\in\mI}x^{\sharp(\lambda)}q^{|\lambda|}.$$

\medskip

Assume that $\Pi=\{\pi_1,\pi_2,\ldots,\pi_K\}$ where $\pi_1=\emptyset$, the empty partition. We define a $(0,1)$-matrix $\mA=\mA(\langle\Pi,\cL\rangle)$ by
\begin{equation}\label{eq:adj-p}
\mA_{i,j}=\begin{cases}
1 & \text{if $\pi_j\in\cL(\pi_i)$},\\
0 & \text{if $\pi_j\not\in\cL(\pi_i)$},
\end{cases}
\end{equation}
and a diagonal matrix $\mW(x)=\mW(\langle\Pi,\cL\rangle\,|\,x,q)$ by
\begin{equation}\label{eq:w-mat-ptn}
\mW(x)=\begin{pmatrix}
x^{\sharp(\pi_1)}q^{|\pi_1|}\\
& x^{\sharp(\pi_2)}q^{|\pi_2|}\\
& & \ddots\\
& & & x^{\sharp(\pi_K)}q^{|\pi_K|}
\end{pmatrix}.
\end{equation}

Let the \textit{$S$-tail} of a partition $\lambda$ be the collection of parts $\le S$ in $\lambda$.

\begin{theorem}\label{th:main}
	For each $1\le k\le K$, we denote by $\mI_k$ the subset of partitions $\lambda$ in $\mI(\langle\Pi,\cL\rangle,S)$ whose $S$-tail is $\pi_k\in\Pi$. We further write
	$$\mG_k(x)=\mG_k(x,q):=\sum_{\lambda\in\mI_k}x^{\sharp(\lambda)}q^{|\lambda|}.$$
	Let $\mA$ and $\mW(x)$ be defined as in \eqref{eq:adj-p} and \eqref{eq:w-mat-ptn}, respectively. Then, for $|q|<1$ and $|x|<|q|^{-1}$,
	\begin{equation}\label{eq:main}
	\begin{pmatrix}
	\mG_1(x)\\
	\mG_2(x)\\
	\vdots\\
	\mG_K(x)
	\end{pmatrix}
	=\mW(x).\left(\lim_{M\to\infty}\prod_{m=1}^M (\mA.\mW(x q^{mS}))\right).\begin{pmatrix}
	1\\
	0\\
	\vdots\\
	0
	\end{pmatrix}.
	\end{equation}
\end{theorem}

\begin{remark}
	Recall that $\pi_1=\emptyset$ (so that $\pi_1\in\cL(\pi)$ for all $\pi\in \Pi$) and $\cL(\emptyset)=\Pi$. It follows that all entries in the first row and column of $\mA$ are $1$. Further, the first entry in $\mW(x)$ is also $x^0q^0=1$. When $|q|<1$ and $|x|<|q|^{-1}$, we have
	$$\lim_{M\to\infty}\mA.\mW(x q^{MS})=\begin{pmatrix}
	1 & 0 & 0 & \cdots & 0\\
	1 & 0 & 0 & \cdots & 0\\
	\vdots & \vdots & \vdots & \ddots & \vdots\\
	1 & 0 & 0 & \cdots & 0
	\end{pmatrix}.$$
	Throughout, $\prod_{m=1}^M (\mA.\mW(x q^{mS}))$ means
	\begin{equation}
	\mA.\mW(xq^S).\mA.\mW(xq^{2S}).\cdots. \mA.\mW(xq^{MS}).
	\end{equation}
\end{remark}

\begin{remark}
	We have
	$$\mG(x)=\sum_{k=1}^K \mG_k(x),$$
	but since $\cL(\emptyset)=\Pi$, it is not hard to see that
	$$\mG_1(x)=\sum_{k=1}^K \mG_k(xq^S).$$
	Hence,
	\begin{equation}
	\mG(x)=\mG_1(xq^{-S}).
	\end{equation}
\end{remark}

\medskip

In September 2018, George Andrews communicated to Zhitai Li and the author a conjecture on the generating function for linked partition ideals, which was recorded in \cite{CL2018}.

\begin{conjecture}[Andrews]
	Every linked partition ideal has a two-variable generating function of the form
	\begin{equation}\label{eq:And-conj-0}
	\left(\begin{array}{c}\text{product}\\\text{of}\\\text{$q$-factorials}\end{array}\right)\times \sum_{n_1,\ldots,n_r\ge 0}\frac{(-1)^{L_1(n_1,\ldots,n_r)}q^{Q(n_1,\ldots,n_r)+L_2(n_1,\ldots,n_r)}x^{L_3(n_1,\ldots,n_r)}}{(q^{A_1};q^{A_1})_{n_1}\cdots (q^{A_r};q^{A_r})_{n_r}},
	\end{equation}
	in which $L_1$, $L_2$ and $L_3$ are linear forms in $n_1,\ldots,n_r$ and $Q$ is a quadratic form in $n_1,\ldots,n_r$. Here the coefficient of the $x^mq^n$ term is the number of partitions of $n$ in this linked partition ideal with $m$ parts.
\end{conjecture}

\medskip

By examining a number of examples in \cite{CL2018,KR2018,Kur2018,Kur2019}, it seems that in some cases the $\mG_k(x)$'s in Theorem \ref{th:main} are of a unified form of $q$-multi-summations. It motivates us to consider a matrix factorization problem involving column functional vectors of certain $q$-multi-summations. This, in turn, provides some crude ideas for the conjecture of Andrews.

Further, the algebraic method in \cite{CL2018} of proving generating function identities such as \eqref{eq:I_1} relies heavily computer algebra (\textit{Mathematica} packages \texttt{qMultiSum} \cite{Rie2003} and \texttt{qGeneratingFunctions} \cite{KK2009}). Now we are able to present a new approach to get rid of such computer assistance. 

\subsection{Outline of this paper}

This paper is organized as follows.

In \S{}\ref{sec:d-graph}, we first define the generating function for walks in a directed graph $G$. Then, by assigning an empty vertex to $G$, we obtain a modified directed graph $G^!$. The generating function of $G^!$ can be defined naturally. Now we merely need to define the associated directed graph of a span one linked partition ideal $\mI(\langle\Pi,\cL\rangle,S)$ and then deduce Theorem \ref{th:main} from the generating function of this associated directed graph.

In \S{}\ref{sec:q-mul-sum}, we will study a $q$-difference system arising from Theorem \ref{th:main}. Two examples will then be discussed: one example comes from the Rogers--Ramanujan identities and the other is about the Kanade--Russell conjectures $I_1$--$I_3$. Then, a matrix factorization problem will be identified from the two examples.

In \S{}\ref{sec:non-com-pf}, we turn to non-computer-assisted proofs of two identities obtained in \S{}\ref{sec:q-mul-sum}. The two identities, in turn, can be used to prove Andrews--Gordon type generating function identities for span one linked partition ideals. Our approach relies on a key recurrence relation obtained in \S{}\ref{sec:rec}. Also, we are able to illustrate the proofs by binary trees.

Finally, we are going to raise some open problems in \S{}\ref{sec:rmks}.

\section{Directed graphs}\label{sec:d-graph}

Let $G=(V,E)$ be a directed graph where $V$ is the set of vertices and $E$ is the set of directed edges. Throughout, we allow loops (that is, directed edges connecting vertices with themselves) in $G$ but for any two vertices $u$ and $v$, not necessarily distinct, we allow at most one directed edge connecting $u$ with $v$. Let $V=\{v_1,v_2,\ldots,v_K\}$. Let $\mA=\mA(G)$ be the adjacency matrix of $G$, that is,
\begin{align}\label{eq:adj-m}
\mA_{i,j}=\begin{cases}
1 & \text{if there is a directed edge from $v_i$ with $v_j$},\\
0 & \text{if there are no directed edges from $v_i$ with $v_j$}.
\end{cases}
\end{align}

We say that $w$ is a walk of step $M$ in $G$ if $w$ is a chain of $M+1$ vertices
$$\varpi_0\to\varpi_1\to\cdots\to\varpi_M$$
such that for each $1\le m\le M$, there is an edge from $\varpi_{m-1}$ to $\varpi_{m}$. Let $\mathcal{W}_M$ be the set of walks of step $M$ in $G$.

\subsection{Generating function for walks in a directed graph}\label{sec:gf-d-graph}

To define the generating function for step $M$ walks in a directed graph $G=(V,E)$, we assign two weights to each vertex $v$: one is called \textit{length}, denoted by $\sharp(v)\in\mathbb{N}$, and the other is called \textit{size}, denoted by $|v|\in\mathbb{N}$.

Let the \textit{shift} $S$ be a non-negative integer.

\medskip

For any walk $w\in\mathcal{W}_M$,
\begin{equation}\label{eq:chain-ver}
w=\varpi_0\to\varpi_1\to\cdots\to\varpi_M,
\end{equation}
we define its generating function by
\begin{align}
\mG(w\,|\, x,q):=x^{\sharp(\varpi_0)}q^{|\varpi_0|}\times (xq^S)^{\sharp(\varpi_1)}q^{|\varpi_1|} \times\cdots\times (xq^{MS})^{\sharp(\varpi_M)}q^{|\varpi_M|}.
\end{align}
Now we are able to define the generating function for step $M$ walks from $v_i$ to $v_j$ for any $1\le i,j\le K$:
\begin{align}
\mG_{i,j}(\mathcal{W}_M\,|\,x)=\mG_{i,j}(\mathcal{W}_M\,|\,x,q):=\sum_{\substack{w\in\mathcal{W}_M\\\varpi_0=v_i\\\varpi_M=v_j}}\mG(w\,|\, x,q).
\end{align}

Let us define a diagonal matrix $\mW(x)=\mW(x,q)$ by
\begin{equation}\label{eq:w-mat-graph}
\mW(x)=\begin{pmatrix}
x^{\sharp(v_1)}q^{|v_1|}\\
& x^{\sharp(v_2)}q^{|v_2|}\\
& & \ddots\\
& & & x^{\sharp(v_K)}q^{|v_K|}
\end{pmatrix}.
\end{equation}

\begin{theorem}\label{th:graph-1}
	Let $\mA$ be the adjacency matrix of $G$ and let $\mW(x)$ be as in \eqref{eq:w-mat-graph}. Then $\mG_{i,j}(\mathcal{W}_M\,|\,x)$ is the $(i,j)$-th entry of
	\begin{equation}\label{eq:graph-1}
	\mW(x).\mA.\mW(xq^S).\mA.\mW(xq^{2S}).\cdots. \mA.\mW(xq^{MS}).
	\end{equation}
\end{theorem}

\begin{remark}
	Let us set $x=q=1$. Then $\mW(1,1)$ is a $K\times K$ identity matrix and hence \eqref{eq:graph-1} becomes $\mA^M$. Since $\mG_{i,j}(\mathcal{W}_M\,|\,1,1)$ equals the number of walks of step $M$ from vertex $v_i$ to vertex $v_j$, Theorem \ref{th:graph-1} immediately leads to a well-known result in graph theory:
\end{remark}

\begin{corollary}
	The number of walks of step $M$ from vertex $v_i$ to vertex $v_j$ is the $(i,j)$-th entry of $\mA^{M}$.
\end{corollary}

\begin{proof}[Proof of Theorem \ref{th:graph-1}]
	We induct on $M$. When $M=0$, that is, the chain $w$ of vertices in \eqref{eq:chain-ver} contains only one vertex $\varpi_0$, it follows that
	$$\mG_{i,j}(\mathcal{W}_0\,|\,x)=\begin{cases}
	x^{\sharp(v_i)}q^{|v_i|} & \text{if $i=j$},\\
	0 & \text{if $i\ne j$},
	\end{cases}$$
	which is identical to the $(i,j)$-th entry of $\mW(x)$.
	
	Now let us assume that the theorem is true for some $M\ge 0$. We also write for convenience
	$$\mM(M)=\mW(x).\mA.\mW(xq^S).\mA.\mW(xq^{2S}).\cdots. \mA.\mW(xq^{MS}).$$
	Then $\mG_{i,j}(\mathcal{W}_M\,|\,x)=\mM(M)_{i,j}$. Further,
	\begin{align*}
	\mM(M+1)_{i,j}&=\sum_{k=1}^K \mM(M)_{i,k}\mA_{k,j}(xq^{(M+1)S})^{\sharp(v_j)}q^{|v_j|}\\
	&=\sum_{k=1}^K \mG_{i,k}(\mathcal{W}_M\,|\,x)\mA_{k,j}(xq^{(M+1)S})^{\sharp(v_j)}q^{|v_j|}.
	\end{align*}
	On the other hand,
	\begin{align*}
	\mG_{i,j}(\mathcal{W}_{M+1}\,|\,x)&=\sum_{\substack{w\in\mathcal{W}_{M+1}\\\varpi_0=v_i\\\varpi_M=v_j}}\mG(w\,|\, x,q)\\
	&=\sum_{k=1}^K \left(\sum_{\substack{w\in\mathcal{W}_{M}\\\varpi_0=v_i\\\varpi_M=v_k}}\mG(w\,|\, x,q)\right)\mA_{k,j}(xq^{(M+1)S})^{\sharp(v_j)}q^{|v_j|}\\
	&=\sum_{k=1}^K \mG_{i,k}(\mathcal{W}_M\,|\,x)\mA_{k,j}(xq^{(M+1)S})^{\sharp(v_j)}q^{|v_j|}.
	\end{align*}
	Hence, $\mG_{i,j}(\mathcal{W}_{M+1}\,|\,x)=\mM(M+1)_{i,j}$, which is our desired result.
\end{proof}

\subsection{Assigning an empty vertex}\label{sec:mod-d-g}

Let us assume that $v_1\in V$ is an empty vertex, that is, its length and size are both $0$:
\begin{equation}
\sharp(v_1)=0 \quad\text{and}\quad |v_1|=0.
\end{equation}
We also assume that, for $2\le k\le K$, $\sharp(v_k)$ and $|v_k|$ are both positive integers.

We require that, for each $1\le k\le K$, there is an edge from vertex $v_k$ to the empty vertex $v_1$. Hence, the entries in the first column of the adjacency matrix $\mA$ are all $1$.

We call such modified directed graph $G^!=(V^!,E^!)$.

\medskip

For any finite walk in $G^!$,
$$w=\varpi_0\to\varpi_1\to\cdots\to\varpi_M,$$
we may extend it to an infinite walk
$$w^\star=\varpi_0\to\varpi_1\to\cdots\to\varpi_M\to v_1\to v_1\to\cdots.$$
It follows from the assumptions $\sharp(v_1)=0$ and $|v_1|=0$ that
\begin{align}
\mG(w^\star\,|\, x,q)=\mG(w\,|\, x,q).
\end{align}
Let $\mathcal{W}^\star$ denote the set of infinite walks in $G^!$ ending with $v_1\to v_1\to\cdots$, a series of empty vertex.

We are now in the position to define the generating function of $G^!$, by
\begin{align}
\mG(G^!\,|\,x,q):=&\sum_{w^\star \in \mathcal{W}^\star}\mG(w^\star\,|\, x,q)\\
=&\sum_{M\ge 0}\sum_{\substack{w\in\mathcal{W}_M\\ w_M\ne v_1}}\mG(w\,|\, x,q).
\end{align}

\begin{theorem}\label{th:graph-inf}
	For each $1\le k\le K$, let $\mG_k(G^!\,|\,x)=\mG_k(G^!\,|\,x,q)$ denote the generating function for infinite walks in $\mathcal{W}^\star$ starting at $v_k$. Let the shift $S$ be a positive integer. Let $\mA$ and $\mW(x)$ be defined as in \eqref{eq:adj-m} and \eqref{eq:w-mat-graph}, respectively. Then, for $|q|<1$ and $|x|<|q|^{-1}$,
	\begin{equation}\label{eq:g-main}
	\begin{pmatrix}
	\mG_1(G^!\,|\,x)\\
	\mG_2(G^!\,|\,x)\\
	\vdots\\
	\mG_K(G^!\,|\,x)
	\end{pmatrix}
	=\mW(x).\left(\lim_{M\to\infty}\prod_{m=1}^M (\mA.\mW(x q^{mS}))\right).\begin{pmatrix}
	1\\
	0\\
	\vdots\\
	0
	\end{pmatrix}.
	\end{equation}
\end{theorem}

\begin{proof}
	We simply observe that, for each $1\le k\le K$, $\mG_k(G^!\,|\,x)$ is the $(k,1)$-th entry of
	$$\mW(x).\left(\lim_{M\to\infty}\prod_{m=1}^M (\mA.\mW(x q^{mS}))\right).$$
	The desired result therefore follows.
\end{proof}

\subsection{Proof of Theorem \ref{th:main}}\label{sec:proof-th-main}

To prove Theorem \ref{th:main}, let us define the \textit{associated directed graph} of a span one linked partition ideal $\mI=\mI(\langle\Pi,\cL\rangle,S)$.

We first define the set of vertices. Since $\Pi=\{\pi_1,\pi_2,\ldots,\pi_K\}$ is a finite set of partitions, we may treat each $\pi_k$ as a vertex. We also define the length of $\pi_k$ as the number of parts in $\pi_k$ and the size of $\pi_k$ as the sum of all parts in $\pi_k$. In particular, since $\pi_1$ is an empty partition so that $\sharp(\pi_1)=0$ and $|\pi_1|=0$, we may treat $\pi_1$ as an empty vertex.

We next define the directed edges in a natural way. For $1\le i,j\le K$, if $\pi_j\in\cL(\pi_i)$, then we say that there is an edge from vertex $\pi_i$ to vertex $\pi_j$. Since $\cL(\pi_1)=\cL(\emptyset)=\Pi$, we know that, for each $1\le k\le K$, there is an edge from vertex $\pi_k$ to vertex $\pi_1$.

We call this graph the associated directed graph of $\mI$, denoted by $G^!(\mI)=(V^!(\mI),E^!(\mI))$. In fact, $G^!(\mI)$ is a modified directed graph described in \S{}\ref{sec:mod-d-g}.

\medskip

Recall from \eqref{eq:decomp} that each partition $\lambda$ in $\mI$ can be uniquely decomposed as
$$\lambda=\lambda_0\oplus\phi^S(\lambda_1)\oplus\phi^{2S}(\lambda_2)\oplus\cdots\oplus \phi^{KS}(\lambda_K)\oplus \phi^{(K+1)S}(\emptyset)\oplus \phi^{(K+2)S}(\emptyset)\oplus\cdots$$
so that $\lambda_K\ne \emptyset$ as long as $\lambda\ne \emptyset$. Hence, we have a natural bijection to infinite walks in $G^!(\mI)$ ending with $\pi_1\to\pi_1\to\cdots$:
$$w^\star(\lambda)=\lambda_0\to\lambda_1\to\lambda_2\to\cdots\to\lambda_K\to \pi_1\to\pi_1\to\cdots.$$
Further, if $\lambda$ is an empty partition, then the resulted infinite walk is simply $\pi_1\to\pi_1\to\cdots$.

Now let us define $S$ to be the shift. Then
\begin{equation}
x^{\sharp(\lambda)}q^{|\lambda|}=\mG(w^\star(\lambda)\,|\, x,q).
\end{equation}
Hence,
$$\mG(x)=\sum_{\lambda\in\mI}x^{\sharp(\lambda)}q^{|\lambda|}=\sum_{w^\star \in \mathcal{W}^\star}\mG(w^\star\,|\, x,q).$$

\medskip

The rest follows directly from Theorem \ref{th:graph-inf}.

\medskip

\begin{example}\label{ex:2.1}
	It is shown in Example \ref{ex:1.1} that partitions with difference at least $2$ at distance $1$ form a span one linked partition ideal $\mI(\langle\Pi,\cL\rangle,S)$ where $\Pi=\{\emptyset,1,2\}$, the linking sets are
	$$\cL(\emptyset)=\{\emptyset,1,2\},\quad \cL(1)=\{\emptyset,1,2\},\quad \cL(2)=\{\emptyset,2\},$$
	and $S=2$. We represent its associated directed graph in Fig.~\ref{fig:a-d-g}.
	
	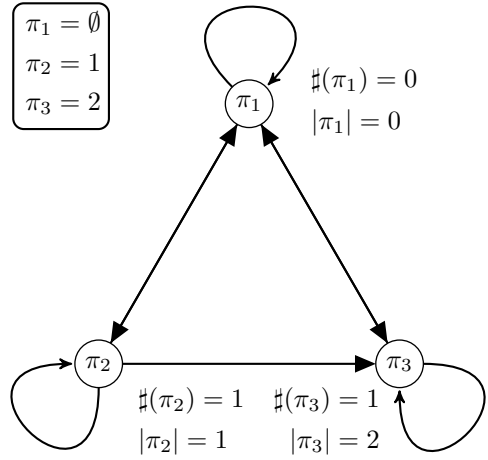
\begin{figure}[ht]
		\caption{The associated directed graph in Example \ref{ex:2.1}}\label{fig:a-d-g}
		\medskip
		\begin{tikzpicture}
		\SetVertexMath
		\Vertex[Math,L=\pi_1,x=2,y=0]{A}
		\Vertex[Math,L=\pi_2,x=0,y=-1.73*2]{B}
		\Vertex[Math,L=\pi_3,x=4,y=-1.73*2]{C}
		\tikzset{EdgeStyle/.style={->,>=triangle 45}}
		\Edge(A)(B)
		\Edge(A)(C)
		\Edge(B)(A)
		\Edge(B)(C)
		\Edge(C)(A)
		\Loop[dir=NO,dist=2cm](A)
		\Loop[dir=SOWE,dist=2cm](B)
		\Loop[dir=SOEA,dist=2cm](C)
		\node (w1) at (3.5,0) {%
			$\begin{aligned}
			& \sharp(\pi_1)=0\\
			& |\pi_1|=0
			\end{aligned}$};
		\node (w2) at (1.2,-4.25) {%
			$\begin{aligned}
			& \sharp(\pi_2)=1\\
			& |\pi_2|=1
			\end{aligned}$};
		\node (w3) at (3,-4.25) {%
			$\begin{aligned}
			\sharp(\pi_3)=1 &\\
			|\pi_3|=2 &
			\end{aligned}$};
		\node[draw=black,thick,rectangle,inner sep=3pt, rounded corners=4pt] (w) at (-0.5,0.5) {%
			$\begin{aligned}
			\pi_1&=\emptyset\\
			\pi_2&=1\\
			\pi_3&=2
			\end{aligned}$};
		\end{tikzpicture}
	\end{figure}
\end{example}

\section{$q$-Multi-summations}\label{sec:q-mul-sum}

\subsection{A $q$-difference system and the uniqueness of solutions}\label{sec:q-diff-sys}

Recall that in Theorem \ref{th:main} we have shown that
\begin{equation}
\begin{pmatrix}
\mG_1(x)\\
\mG_2(x)\\
\vdots\\
\mG_K(x)
\end{pmatrix}
=\mW(x).\left(\lim_{M\to\infty}\prod_{m=1}^M (\mA.\mW(x q^{mS}))\right).\begin{pmatrix}
1\\
0\\
\vdots\\
0
\end{pmatrix}.
\end{equation}
Let us focus on
\begin{equation}
\begin{pmatrix}
F_1^\star(x)\\
F_2^\star(x)\\
\vdots\\
F_K^\star(x)
\end{pmatrix}
:=\left(\lim_{M\to\infty}\prod_{m=1}^M (\mA.\mW(x q^{mS}))\right).\begin{pmatrix}
1\\
0\\
\vdots\\
0
\end{pmatrix}.
\end{equation}
Notice that
\begin{align*}
\begin{pmatrix}
F_1^\star(x)\\
F_2^\star(x)\\
\vdots\\
F_K^\star(x)
\end{pmatrix}
&=\left(\lim_{M\to\infty}\prod_{m=1}^M (\mA.\mW(x q^{mS}))\right).\begin{pmatrix}
1\\
0\\
\vdots\\
0
\end{pmatrix}\\
&=
\mA.\mW(xq^S).\left(\lim_{M\to\infty}\prod_{m=1}^M (\mA.\mW(xq^S q^{mS}))\right).\begin{pmatrix}
1\\
0\\
\vdots\\
0
\end{pmatrix}\\
&=\mA.\mW(xq^S).\begin{pmatrix}
F_1^\star(xq^S)\\
F_2^\star(xq^S)\\
\vdots\\
F_K^\star(xq^S)
\end{pmatrix}.
\end{align*}
If we further write $F_k(x):=F_k^\star(xq^{-S})$ for each $k$, then the column vector
$$\underline{\mathbf{F}}(x):=\begin{pmatrix}
F_1(x)\\
F_2(x)\\
\vdots\\
F_K(x)
\end{pmatrix}$$
satisfies the $q$-difference system
\begin{equation}\label{eq:q-diff-sys}
\underline{\mathbf{F}}(x)=\mA.\mW(x).\underline{\mathbf{F}}(xq^S).
\end{equation}

\medskip

\begin{remark}\label{re:F_0}
It follows from \eqref{eq:q-diff-sys} that
\begin{equation}\label{eq:F-G}
\underline{\mathbf{F}}(x)=\mA.\mW(x).\begin{pmatrix}
F_1^\star(x)\\
F_2^\star(x)\\
\vdots\\
F_K^\star(x)
\end{pmatrix}=\mA.\begin{pmatrix}
\mG_1(x)\\
\mG_2(x)\\
\vdots\\
\mG_K(x)
\end{pmatrix}.
\end{equation}
Recall that, we have defined in Theorem \ref{th:main} that, for each $1\le k\le K$, $\mI_k$ denotes the subset of partitions in $\mI(\langle\Pi,\cL\rangle,S)$ whose $S$-tail is $\pi_k$. Further, $\mG_k(x)$ is the generating function of $\mI_k$. Since $\mA$ is a $(0,1)$-matrix, it follows that $F_k(x)\in\mathbb{Z}[[q]][[x]]$ for each $1\le k\le K$. More importantly, since the empty partition $\emptyset$ is contained in $\mI_1$ but not in $\mI_k$ for $2\le k\le K$, we have $\mG_1(0)=1$ and $\mG_k(0)=0$ for $2\le k\le K$. Since the entries in the first column of $\mA$ are all $1$, it follows that
\begin{equation}
F_1(0)=F_2(0)=\cdots=F_K(0)=1.
\end{equation}
\end{remark}

\medskip

We next show the uniqueness of solutions of \eqref{eq:q-diff-sys}.

\begin{proposition}\label{prop:uni}
	In the $q$-difference system \eqref{eq:q-diff-sys}, we assume that, for each $1\le k\le K$, $F_k(x)\in\mathbb{C}[[q]][[x]]$. If $F_1(0)=F_2(0)=\cdots=F_K(0)$, then there exists a solution to \eqref{eq:q-diff-sys}. Further, the solution is uniquely determined by $\underline{\mathbf{F}}(0)$.
\end{proposition}

\begin{proof}
	For each $1\le k\le K$, let us write
	$$F_k(x)=\sum_{n\ge 0}f_k(n)x^n,$$
	where $f_k(n)\in\mathbb{C}[[q]]$ for $n\ge 0$. We also write for notational convenience that $f_k(n)=0$ for $n<0$. Then,
	\begin{align*}
	\sum_{n\ge 0}f_k(n)x^n &= \sum_{j=1}^K \mA_{k,j}x^{\sharp(\pi_j)}q^{|\pi_j|}\sum_{n\ge 0}f_j(n)q^{nS}x^n\\
	&=\sum_{n\ge 0}\left(\sum_{j=1}^K \mA_{k,j}q^{|\pi_j|+(n-\sharp(\pi_j))S}f_j(n-\sharp(\pi_j))\right)x^n.
	\end{align*}
	Recall that $\sharp(\pi_1)=|\pi_1|=0$ and $\mA_{k,1}=1$ for all $k$. We have that, for $n\ge 0$,
	\begin{align}\label{eq:f_k(n)}
	f_k(n)=q^{nS}f_1(n)+\sum_{j=2}^K \mA_{k,j}q^{|\pi_j|+(n-\sharp(\pi_j))S}f_j(n-\sharp(\pi_j)).
	\end{align}
	Setting $n=0$ gives the requirement $F_1(0)=F_2(0)=\cdots=F_K(0)$. Also, $\underline{\mathbf{F}}(0)=(f_1(0),f_2(0),\ldots,f_K(0))^{T}$ uniquely determines $f_k(n)$ for all $1\le k\le K$ and $n\ge 1$ by \eqref{eq:f_k(n)}.
\end{proof}

\subsection{Two examples}\label{sec:ex}

Recall that, for each $1\le k\le K$, $\mI_k$ denotes the subset of partitions in $\mI(\langle\Pi,\cL\rangle,S)$ whose $S$-tail is $\pi_k$. Further,
$$\mG_k(x)=\sum_{\lambda\in\mI_k}x^{\sharp(\lambda)}q^{|\lambda|}.$$

\subsubsection{Example 1}\label{sec:ex1}

In the first example, we consider
\begin{center}
	``partitions with difference at least $2$ at distance $1$.''
\end{center}
This partition set obviously corresponds to the Rogers--Ramanujan identities. In Example \ref{ex:1.1}, we have shown that it is a span one linked partition ideal $\mI(\langle\Pi,\cL\rangle,S)$ where $\Pi=\{\pi_1,\pi_2,\pi_3\}$ with $\pi_1=\emptyset$, $\pi_2=1$ and $\pi_3=2$, the linking sets are
$$\cL(\pi_1)=\{\pi_1,\pi_2,\pi_3\},\quad \cL(\pi_2)=\{\pi_1,\pi_2,\pi_3\},\quad \cL(\pi_3)=\{\pi_1,\pi_3\},$$
and $S=2$.

Notice that the generating function for partitions with difference at least $2$ at distance $1$ is
\begin{equation}\label{eq:ex-1-g1}
\mG_1(x)+\mG_2(x)+\mG_3(x)=\sum_{n\ge 0}\frac{q^{n^2}x^n}{(q;q)_n}
\end{equation}
and that the generating function for partitions with difference at least $2$ at distance $1$ with the smallest part $\ge 2$ is
\begin{equation}\label{eq:ex-1-g3}
\mG_1(x)+\mG_3(x)=\sum_{n\ge 0}\frac{q^{n^2+n}x^n}{(q;q)_n}.
\end{equation}

We know from \eqref{eq:F-G} that
$$\begin{pmatrix}
F_1(x)\\
F_2(x)\\
F_3(x)
\end{pmatrix}=\mA.\begin{pmatrix}
\mG_1(x)\\
\mG_2(x)\\
\mG_3(x)
\end{pmatrix}
=\begin{pmatrix}
1 & 1 & 1\\
1 & 1 & 1\\
1 & 0 & 1
\end{pmatrix}.\begin{pmatrix}
\mG_1(x)\\
\mG_2(x)\\
\mG_3(x)
\end{pmatrix}.$$
Hence, by \eqref{eq:ex-1-g1} and \eqref{eq:ex-1-g3}, if we put
\begin{gather}
F_1(x)=F_2(x)=\sum_{n\ge 0}\frac{q^{n^2}x^n}{(q;q)_n}\\
\intertext{and}
F_3(x)=\sum_{n\ge 0}\frac{q^{n^2+n}x^n}{(q;q)_n},
\end{gather}
then we have the following relation from \eqref{eq:q-diff-sys}:
\begin{equation}\label{eq:q-diff-sys-ex1}
\begin{pmatrix}
F_1(x)\\
F_2(x)\\
F_3(x)
\end{pmatrix}=\begin{pmatrix}
1 & 1 & 1\\
1 & 1 & 1\\
1 & 0 & 1
\end{pmatrix}.\begin{pmatrix}
1\\
& xq\\
& & xq^2
\end{pmatrix}.\begin{pmatrix}
F_1(xq^2)\\
F_2(xq^2)\\
F_3(xq^2)
\end{pmatrix}
\end{equation}

Conversely, if we are able to prove \eqref{eq:q-diff-sys-ex1} directly (notice that $F_1(0)=F_2(0)=F_3(0)=1$), then by Remark \ref{re:F_0} and Proposition \ref{prop:uni}, we can compute that
\begin{align*}
\begin{pmatrix}
	\mG_1(x)\\
	\mG_2(x)\\
	\mG_3(x)
\end{pmatrix}&=
\begin{pmatrix}
1\\
& xq\\
& & xq^2
\end{pmatrix}.\begin{pmatrix}
F_1^\star(x)\\
F_2^\star(x)\\
F_3^\star(x)
\end{pmatrix}\\
&=\begin{pmatrix}
1\\
& xq\\
& & xq^2
\end{pmatrix}.\begin{pmatrix}
F_1(xq^2)\\
F_2(xq^2)\\
F_3(xq^2)
\end{pmatrix}.
\end{align*}
Also, \eqref{eq:ex-1-g1} and \eqref{eq:ex-1-g3} can be deduced with no difficulty.

\subsubsection{Example 2}\label{sec:ex2}

In the second example, we consider
\begin{quote}
	``partitions with difference at least $3$ at distance $2$ such that if two consecutive parts differ by at most $1$, then their sum is divisible by $3$.''
\end{quote}
This partition set corresponds to the Kanade--Russell conjectures $I_1$--$I_3$. It was shown in \cite{CL2018} that this partition set is a span one linked partition ideal $\mI(\langle\Pi,\cL\rangle,S)$ where $S=3$, and $\Pi=\{\pi_1,\pi_2,\ldots,\pi_7\}$ along with the linking sets are given as follows.
\begin{equation*}
\begin{array}{cp{0.5cm}c}
\Pi && \text{linking set}\\
\pi_1 = \emptyset && \{\pi_1,\; \pi_2,\; \pi_3,\; \pi_4,\; \pi_5,\; \pi_6,\; \pi_7\}\\
\pi_2 = 1 && \{\pi_1,\; \pi_2,\; \pi_3,\; \pi_4,\; \pi_5,\; \pi_6,\; \pi_7\}\\
\pi_3 = 2+1 && \{\pi_1,\; \pi_2,\; \pi_3,\; \pi_4,\; \pi_5,\; \pi_6,\; \pi_7\}\\
\pi_4 = 3+1 && \{\pi_1,\; \pi_5,\; \pi_6,\; \pi_7\}\\
\pi_5 = 2 && \{\pi_1,\; \pi_2,\; \pi_3,\; \pi_4,\; \pi_5,\; \pi_6,\; \pi_7\}\\
\pi_6 = 3 && \{\pi_1,\; \pi_5,\; \pi_6,\; \pi_7\}\\
\pi_7 = 3+3 && \{\pi_1,\; \pi_6,\; \pi_7\}
\end{array}
\end{equation*}

It was also shown in \cite{CL2018} that the generating function for such partitions is
\begin{equation}
\begin{aligned}
\mG_1(x)&+\mG_2(x)+\mG_3(x)+\mG_4(x)\\
&+\mG_5(x)+\mG_6(x)+\mG_7(x)
\end{aligned}
=\sum_{n_1,n_2\ge 0}\frac{q^{n_1^2+3n_2^2+3n_1n_2}x^{n_1+2n_2}}{(q;q)_{n_1} (q^3;q^3)_{n_2}},\label{eq:ex-2-t11}
\end{equation}
that the generating function for such partitions with the smallest part $\ge 2$ is
\begin{equation}
\mG_1(x)+\mG_5(x)+\mG_6(x)+\mG_7(x)
=\sum_{n_1,n_2\ge 0}\frac{q^{n_1^2+3n_2^2+3n_1n_2+n_1+3n_2}x^{n_1+2n_2}}{(q;q)_{n_1} (q^3;q^3)_{n_2}},\label{eq:ex-2-t12}
\end{equation}
and that the generating function for such partitions with the smallest part $\ge 3$ is
\begin{equation}
\mG_1(x)+\mG_6(x)+\mG_7(x)
=\sum_{n_1,n_2\ge 0}\frac{q^{n_1^2+3n_2^2+3n_1n_2+2n_1+3n_2}x^{n_1+2n_2}}{(q;q)_{n_1} (q^3;q^3)_{n_2}}.\label{eq:ex-2-t13}
\end{equation}

We know from \eqref{eq:F-G} that
$$\begin{pmatrix}
F_1(x)\\
F_2(x)\\
F_3(x)\\
F_4(x)\\
F_5(x)\\
F_6(x)\\
F_7(x)
\end{pmatrix}
=\begin{pmatrix}
1 & 1 & 1 & 1 & 1 & 1 & 1\\
1 & 1 & 1 & 1 & 1 & 1 & 1\\
1 & 1 & 1 & 1 & 1 & 1 & 1\\
1 & 0 & 0 & 0 & 1 & 1 & 1\\
1 & 1 & 1 & 1 & 1 & 1 & 1\\
1 & 0 & 0 & 0 & 1 & 1 & 1\\
1 & 0 & 0 & 0 & 0 & 1 & 1\\
\end{pmatrix}.\begin{pmatrix}
\mG_1(x)\\
\mG_2(x)\\
\mG_3(x)\\
\mG_4(x)\\
\mG_5(x)\\
\mG_6(x)\\
\mG_7(x)
\end{pmatrix}.$$
Hence, by \eqref{eq:ex-2-t11}, \eqref{eq:ex-2-t12} and \eqref{eq:ex-2-t13}, if we put
\begin{gather}
F_1(x)=F_2(x)=F_3(x)=F_5(x)=\sum_{n_1,n_2\ge 0}\frac{q^{n_1^2+3n_2^2+3n_1n_2}x^{n_1+2n_2}}{(q;q)_{n_1} (q^3;q^3)_{n_2}},\\
F_4(x)=F_6(x)=\sum_{n_1,n_2\ge 0}\frac{q^{n_1^2+3n_2^2+3n_1n_2+n_1+3n_2}x^{n_1+2n_2}}{(q;q)_{n_1} (q^3;q^3)_{n_2}}
\intertext{and}
F_7(x)=\sum_{n_1,n_2\ge 0}\frac{q^{n_1^2+3n_2^2+3n_1n_2+2n_1+3n_2}x^{n_1+2n_2}}{(q;q)_{n_1} (q^3;q^3)_{n_2}},
\end{gather}
then we have the following relation from \eqref{eq:q-diff-sys}:
\begin{equation}\label{eq:q-diff-sys-ex2}
\scalebox{0.865}{%
$\begin{pmatrix}
F_1(x)\\
F_2(x)\\
F_3(x)\\
F_4(x)\\
F_5(x)\\
F_6(x)\\
F_7(x)
\end{pmatrix}
=\begin{pmatrix}
1 & 1 & 1 & 1 & 1 & 1 & 1\\
1 & 1 & 1 & 1 & 1 & 1 & 1\\
1 & 1 & 1 & 1 & 1 & 1 & 1\\
1 & 0 & 0 & 0 & 1 & 1 & 1\\
1 & 1 & 1 & 1 & 1 & 1 & 1\\
1 & 0 & 0 & 0 & 1 & 1 & 1\\
1 & 0 & 0 & 0 & 0 & 1 & 1\\
\end{pmatrix}.\begin{pmatrix}
1\\
& xq\\
&& x^2q^3\\
&&& x^2q^4\\
&&&& xq^2\\
&&&&& xq^3\\
&&&&&& x^2q^6
\end{pmatrix}.\begin{pmatrix}
F_1(xq^3)\\
F_2(xq^3)\\
F_3(xq^3)\\
F_4(xq^3)\\
F_5(xq^3)\\
F_6(xq^3)\\
F_7(xq^3)
\end{pmatrix}$}.
\end{equation}

Conversely, we are also able to recover
$$(\mG_1(x),\mG_2(x),\mG_3(x),\mG_4(x),\mG_5(x),\mG_6(x),\mG_7(x))^T$$
as well as \eqref{eq:ex-2-t11}, \eqref{eq:ex-2-t12} and \eqref{eq:ex-2-t13} provided that we have proved \eqref{eq:q-diff-sys-ex1} directly since $F_1(0)=F_2(0)=\cdots=F_7(0)=1$.

\subsection{A matrix factorization problem}\label{sec:frac}

Motivated by \eqref{eq:q-diff-sys-ex1} and \eqref{eq:q-diff-sys-ex2}, we turn our interest to a matrix factorization problem as follows.

\medskip

Let $R$ be a positive integer. Let $\underline{\boldsymbol{\alpha}}=(\alpha_{i,j})\in\Mat_{R\times R}(\mathbb{N})$ be a fixed symmetric matrix. Let $\underline{\mathbf{A}}=(A_r)\in \mathbb{N}_{>0}^R$ and $\underline{\boldsymbol{\gamma}}=(\gamma_r)\in \mathbb{N}_{>0}^R$ be fixed.

Let $\fF$ be a set of $q$-multi-summations defined by
\begin{align}
\fF:=\big\{H(\ubb)\,:\, \ubb\in\mathbb{Z}^R\ \text{and condition \eqref{eq:condition} is satisfied}\big\},
\end{align}
where $H(\ubb)=H(\beta_1,\ldots,\beta_R)$ is of the form
\begin{align}\label{eq:F-beta}
H(\ubb):=\sum_{n_1,\ldots,n_R\ge 0}\frac{q^{\sum_{r=1}^R \alpha_{r,r}n_r(n_r-1)/2}q^{\sum_{1\le i< j\le R}\alpha_{i,j}n_i n_j}q^{\sum_{r=1}^R \beta_r n_r}x^{\sum_{r=1}^R \gamma_r n_r}}{(q^{A_1};q^{A_1})_{n_1}\cdots (q^{A_R};q^{A_R})_{n_R}}
\end{align}
and the additional condition reads: for all $(n_1,\ldots,n_R)\in \mathbb{N}^R\backslash \{(0,0,\ldots,0)\}$,
\begin{align}\label{eq:condition}
&\sum_{r=1}^R \frac{\alpha_{r,r}n_r(n_r-1)}{2}+\sum_{1\le i< j\le R}\alpha_{i,j}n_i n_j+\sum_{r=1}^R \beta_r n_r>0.
\end{align}

\medskip

Now we consider a column functional vector
\begin{equation}
\ubF_{\ubb}(x)=\begin{pmatrix}
F_1(x)\\
F_2(x)\\
\vdots\\
F_K(x)
\end{pmatrix}:=
\begin{pmatrix}
H(\ubb_1)\\
H(\ubb_2)\\
\vdots\\
H(\ubb_K)
\end{pmatrix},
\end{equation}
where $H(\ubb_k)\in\fF$ for all $1\le k\le K$.

We expect $\ubF_{\ubb}(x)$ to satisfy the following factorization property.

\medskip

\noindent \textbf{Factorization Property.} Let $\mU$ be a $(0,1)$-matrix such that all entries in the first row and column are $1$. Let $\mV$ be a diagonal matrix such that all (diagonal) entries are monic monomials in $x$ and $q$ with $\mV_{1,1}=1$. We say that $\ubF_{\ubb}(x)$ satisfies the \textit{Factorization Property} if
\begin{equation}\label{eq:Fac-Prop}
\ubF_{\ubb}(x) = \mU.\mV.\ubF_{\ubb}(xq^S)
\end{equation}
for some positive integer $S$.

\medskip

\begin{example}\label{ex:1-1}
	In the example in \S{}\ref{sec:ex1}, we have $\uba=\begin{pmatrix}
	2
	\end{pmatrix}$, $\ubg=(1)$, $\ubA=(1)$ and
	$$\ubF_{\ubb}(x)=
	\begin{pmatrix}
	H(1)\\
	H(1)\\
	H(2)
	\end{pmatrix}.$$
	Also, $S=2$.
\end{example}

\begin{example}\label{ex:2-2}
	In the example in \S{}\ref{sec:ex2}, we have $\uba=\begin{pmatrix}
	2 & 3\\3 & 6
	\end{pmatrix}$, $\ubg=(1,2)$, $\ubA=(1,3)$ and
	$$\ubF_{\ubb}(x)=
	\begin{pmatrix}
	H(1,3)\\
	H(1,3)\\
	H(1,3)\\
	H(2,6)\\
	H(1,3)\\
	H(2,6)\\
	H(3,6)
	\end{pmatrix}.$$
	Also, $S=3$.
\end{example}

\section{Non-computer-assisted proofs}\label{sec:non-com-pf}

In \cite{CL2018}, Li and the author provided an algebraic method to prove Andrews--Gordon type generating function identities such as \eqref{eq:ex-2-t11}, \eqref{eq:ex-2-t12} and \eqref{eq:ex-2-t13}. However, one defect in that work is that the proofs rely heavily on computer assistance. Our aim here is to overcome this problem.

As we have seen in \S{}\ref{sec:ex2}, to prove \eqref{eq:ex-2-t11}, \eqref{eq:ex-2-t12} and \eqref{eq:ex-2-t13}, it suffices to show \eqref{eq:q-diff-sys-ex2}.

Our starting point is a recurrence relation enjoyed by $H(\beta_1,\ldots,\beta_R)$ defined in \eqref{eq:F-beta}.

\subsection{A recurrence relation}\label{sec:rec}

Recall that
\begin{align*}
&H(\beta_1,\ldots,\beta_R)\\
&\quad=\sum_{n_1,\ldots,n_R\ge 0}\frac{q^{\sum_{r=1}^R \alpha_{r,r}n_r(n_r-1)/2}q^{\sum_{1\le i< j\le R}\alpha_{i,j}n_i n_j}q^{\sum_{r=1}^R \beta_r n_r}x^{\sum_{r=1}^R \gamma_r n_r}}{(q^{A_1};q^{A_1})_{n_1}\cdots (q^{A_R};q^{A_R})_{n_R}}.
\end{align*}

\medskip
\begin{theorem}\label{th:rec}
	For $1\le r\le R$, we have
	\begin{align}
	H(\beta_1,\ldots,\beta_r,\ldots,\beta_R)&=H(\beta_1,\ldots,\beta_r+A_r,\ldots,\beta_R)\notag\\
	&+x^{\gamma_r}q^{\beta_r}H(\beta_1+\alpha_{r,1},\ldots,\beta_r+\alpha_{r,r},\ldots,\beta_R+\alpha_{r,R}).
	\end{align}
\end{theorem}

\begin{proof}
	We have (recall that $\uba$ is a symmetric matrix so that $\alpha_{i,j}=\alpha_{j,i}$ for $1\le i,j\le R$)
	\begin{align*}
	&H(\beta_1,\ldots,\beta_r,\ldots,\beta_R)-H(\beta_1,\ldots,\beta_r+A_r,\ldots,\beta_R)\\
	&\quad=\sum_{n_1,\ldots,n_R\ge 0}\frac{q^{\sum_i \alpha_{i,i}n_i(n_i-1)/2}q^{\sum_{i<j}\alpha_{i,j}n_i n_j}q^{\sum_i \beta_i n_i}(1-q^{n_rA_r})x^{\sum_i \gamma_i n_i}}{(q^{A_1};q^{A_1})_{n_1}\cdots(q^{A_r};q^{A_r})_{n_r}\cdots (q^{A_R};q^{A_R})_{n_R}}\\
	&\quad=\sum_{\substack{n_1,\ldots,n_R\ge 0\\n_r\ge 1}}\frac{q^{\sum_i \alpha_{i,i}n_i(n_i-1)/2}q^{\sum_{i<j}\alpha_{i,j}n_i n_j}q^{\sum_i \beta_i n_i}x^{\sum_i \gamma_i n_i}}{(q^{A_1};q^{A_1})_{n_1}\cdots(q^{A_r};q^{A_r})_{n_r-1}\cdots (q^{A_R};q^{A_R})_{n_R}}\\
	&\quad=x^{\gamma_r}q^{\beta_r}\sum_{n_1,\ldots,n_R\ge 0}\frac{q^{\sum_i \alpha_{i,i}n_i(n_i-1)/2}q^{\sum_{i<j}\alpha_{i,j}n_i n_j}q^{\sum_i (\beta_i+\alpha_{r,i}) n_i}x^{\sum_i \gamma_i n_i}}{(q^{A_1};q^{A_1})_{n_1}\cdots(q^{A_r};q^{A_r})_{n_r}\cdots (q^{A_R};q^{A_R})_{n_R}}\\
	&\quad=x^{\gamma_r}q^{\beta_r}H(\beta_1+\alpha_{r,1},\ldots,\beta_r+\alpha_{r,r},\ldots,\beta_R+\alpha_{r,R}).
	\end{align*}
	The desired identity therefore follows.
\end{proof}

\medskip

Recall that the Factorization Property says that
$$\ubF_{\ubb}(x) = \mU.\mV.\ubF_{\ubb}(xq^S).$$
Further, if $F(x)=H(\beta_1,\ldots,\beta_R)$, then
\begin{equation}\label{eq:shift-S}
F(xq^S)=H(\beta_1+\gamma_1 S,\ldots,\beta_R+\gamma_R S).
\end{equation}

\medskip

Perhaps, if we expect to apply Theorem \ref{th:rec} to deduce Andrews--Gordon type generating function identities, we need to attach some additional conditions to the Factorization Property.

\medskip

\noindent \textbf{Additional Conditions.} For all $1\le s\le R$:
\begin{enumerate}[label={\textup{(\roman*).}}]
	\item $\gamma_s S\in A_s\mathbb{Z}$;
	\item for all $1\le r\le R$, $\alpha_{r,s}\in A_s\mathbb{Z}$.
\end{enumerate}

\subsection{Proof of (\ref{eq:q-diff-sys-ex1})}\label{sec:pf-1}

We first prove \eqref{eq:q-diff-sys-ex1}, which is relatively easy.

\begin{theorem}
	Let
	\begin{gather}
	F_1(x)=F_2(x)=\sum_{n\ge 0}\frac{q^{n^2}x^n}{(q;q)_n}\\
	\intertext{and}
	F_3(x)=\sum_{n\ge 0}\frac{q^{n^2+n}x^n}{(q;q)_n}.
	\end{gather}
	Then,
	\begin{equation}\label{eq:q-diff-sys-ex1-new}
	\begin{pmatrix}
	F_1(x)\\
	F_2(x)\\
	F_3(x)
	\end{pmatrix}=\begin{pmatrix}
	1 & 1 & 1\\
	1 & 1 & 1\\
	1 & 0 & 1
	\end{pmatrix}.\begin{pmatrix}
	1\\
	& xq\\
	& & xq^2
	\end{pmatrix}.\begin{pmatrix}
	F_1(xq^2)\\
	F_2(xq^2)\\
	F_3(xq^2)
	\end{pmatrix}
	\end{equation}
\end{theorem}

\medskip

We have shown in Example \ref{ex:1-1} that in this case $S=2$, $\uba=\begin{pmatrix}
2
\end{pmatrix}$, $\ubg=(1)$, $\ubA=(1)$ and
$$\begin{pmatrix}
F_1(x)\\
F_2(x)\\
F_3(x)
\end{pmatrix}=
\begin{pmatrix}
H(1)\\
H(1)\\
H(2)
\end{pmatrix}.$$
Further, it follows from \eqref{eq:shift-S} that
\begin{gather}
F_1(xq^2)=F_2(xq^2)=H(3)\\
\intertext{and}
F_3(xq^2)=H(4).
\end{gather}

\medskip

To prove \eqref{eq:q-diff-sys-ex1-new}, it suffices to show that
\begin{gather}
F_1(x)=F_1(xq^2)+xqF_2(xq^2)+xq^2F_3(xq^2)\label{eq:ex-1-1}\\
\intertext{and}
F_3(x)=F_1(xq^2)+xq^2F_3(xq^2).\label{eq:ex-1-2}
\end{gather}

It follows from Theorem \ref{th:rec} that
\begin{align*}
F_1(x)&=H(1)\\
&=H(1+1)+xqH(1+2)\\
&=H(2)+xqH(3)\\
&=\big(H(2+1)+xq^2H(2+2)\big)+xqH(3)\\
&=H(3)+xq^2H(4)+xqH(3)\\
&=F_1(xq^2)+xq^2F_3(xq^2)+xqF_2(xq^2).
\end{align*}
Also,
\begin{align*}
F_3(x)&=H(2)\\
&=H(2+1)+xq^2H(2+2)\\
&=H(3)+xq^2H(4)\\
&=F_1(xq^2)+xq^2F_3(xq^2).
\end{align*}
Identities \eqref{eq:ex-1-1} and \eqref{eq:ex-1-2} are therefore proved.

\subsection{Proof of (\ref{eq:q-diff-sys-ex2})}\label{sec:pf-2}

We next prove \eqref{eq:q-diff-sys-ex2}.

\begin{theorem}
	Let
	\begin{gather}
	F_1(x)=F_2(x)=F_3(x)=F_5(x)=\sum_{n_1,n_2\ge 0}\frac{q^{n_1^2+3n_2^2+3n_1n_2}x^{n_1+2n_2}}{(q;q)_{n_1} (q^3;q^3)_{n_2}},\\
	F_4(x)=F_6(x)=\sum_{n_1,n_2\ge 0}\frac{q^{n_1^2+3n_2^2+3n_1n_2+n_1+3n_2}x^{n_1+2n_2}}{(q;q)_{n_1} (q^3;q^3)_{n_2}}
	\intertext{and}
	F_7(x)=\sum_{n_1,n_2\ge 0}\frac{q^{n_1^2+3n_2^2+3n_1n_2+2n_1+3n_2}x^{n_1+2n_2}}{(q;q)_{n_1} (q^3;q^3)_{n_2}}.
	\end{gather}
	Then,
	\begin{equation}\label{eq:q-diff-sys-ex2-new}
	\scalebox{0.865}{%
		$\begin{pmatrix}
		F_1(x)\\
		F_2(x)\\
		F_3(x)\\
		F_4(x)\\
		F_5(x)\\
		F_6(x)\\
		F_7(x)
		\end{pmatrix}
		=\begin{pmatrix}
		1 & 1 & 1 & 1 & 1 & 1 & 1\\
		1 & 1 & 1 & 1 & 1 & 1 & 1\\
		1 & 1 & 1 & 1 & 1 & 1 & 1\\
		1 & 0 & 0 & 0 & 1 & 1 & 1\\
		1 & 1 & 1 & 1 & 1 & 1 & 1\\
		1 & 0 & 0 & 0 & 1 & 1 & 1\\
		1 & 0 & 0 & 0 & 0 & 1 & 1\\
		\end{pmatrix}.\begin{pmatrix}
		1\\
		& xq\\
		&& x^2q^3\\
		&&& x^2q^4\\
		&&&& xq^2\\
		&&&&& xq^3\\
		&&&&&& x^2q^6
		\end{pmatrix}.\begin{pmatrix}
		F_1(xq^3)\\
		F_2(xq^3)\\
		F_3(xq^3)\\
		F_4(xq^3)\\
		F_5(xq^3)\\
		F_6(xq^3)\\
		F_7(xq^3)
		\end{pmatrix}$}.
	\end{equation}
\end{theorem}

\medskip

We have shown in Example \ref{ex:2-2} that in this case $S=3$, $\uba=\begin{pmatrix}
2 & 3\\3 & 6
\end{pmatrix}$, $\ubg=(1,2)$, $\ubA=(1,3)$ and
$$\begin{pmatrix}
F_1(x)\\
F_2(x)\\
F_3(x)\\
F_4(x)\\
F_5(x)\\
F_6(x)\\
F_7(x)
\end{pmatrix}=
\begin{pmatrix}
H(1,3)\\
H(1,3)\\
H(1,3)\\
H(2,6)\\
H(1,3)\\
H(2,6)\\
H(3,6)
\end{pmatrix}.$$
Again, it follows from \eqref{eq:shift-S} that
\begin{gather}
F_1(xq^3)=F_2(xq^3)=F_3(xq^3)=F_5(xq^3)=H(4,9),\\
F_4(xq^3)=F_6(xq^3)=H(5,12)
\intertext{and}
F_7(xq^3)=H(6,12).
\end{gather}

\medskip

To prove \eqref{eq:q-diff-sys-ex1-new}, it suffices to show that
\begin{gather}
F_1(x)=\left\{\begin{aligned}
&F_1(xq^3)+xqF_2(xq^3)+x^2q^3F_3(xq^3)+x^2q^4F_4(xq^3)\\
&+xq^2F_5(xq^3)+xq^3F_6(xq^3)+x^2q^6F_7(xq^3)
\end{aligned}\right\},\label{eq:ex-2-1}\\
F_4(x)=F_1(xq^3)+xq^2F_5(xq^3)+xq^3F_6(xq^3)+x^2q^6F_7(xq^3)\label{eq:ex-2-2}
\intertext{and}
F_7(x)=F_1(xq^2)+xq^3F_6(xq^3)+x^2q^6F_7(xq^3).\label{eq:ex-2-3}
\end{gather}

We will adopt the following notation to make our argument more transparent. First, a bold term indicates that we will apply Theorem \ref{th:rec} to this term. Also, we will italicize one coordinate if Theorem \ref{th:rec} is applied to that coordinate. Finally, the two underlined terms in the next line are deduced by the previous bold term.

It follows from Theorem \ref{th:rec} that
\begin{align*}
F_1(x)&=\boldsymbol{H(1,\mathit{3})}\\
&=\uwave{\boldsymbol{H(\mathit{1},6)}}+\uwave{x^2q^3H(4,9)}\\
&=\uwave{\boldsymbol{H(\mathit{2},6)}}+\uwave{xq H(3,9)}+x^2q^3H(4,9)\\
&=\uwave{H(3,6)}+\uwave{xq^2H(4,9)}+\boldsymbol{xqH(\mathit{3},9)}+x^2q^3H(4,9)\\
&=\boldsymbol{H(3,\mathit{6})}+xq^2H(4,9)+\uwave{xqH(4,9)}+\uwave{x^2q^4H(5,12)}+x^2q^3H(4,9)\\
&=\uwave{\boldsymbol{H(\mathit{3},9)}}+\uwave{x^2q^6H(6,12)}+xq^2H(4,9)+xqH(4,9)+x^2q^4H(5,12)\\
&\quad+x^2q^3H(4,9)\\
&=\uwave{H(4,9)}+\uwave{xq^3H(5,12)}+x^2q^6H(6,12)+xq^2H(4,9)+xqH(4,9)\\
&\quad+x^2q^4H(5,12)+x^2q^3H(4,9)\\
&=F_1(xq^3)+xq^3F_6(xq^3)+x^2q^6F_7(xq^3)+xq^2F_5(xq^3)+xqF_2(xq^3)\\
&\quad+x^2q^4F_4(xq^3)+x^2q^3F_3(xq^3).
\end{align*}
Also,
\begin{align*}
F_4(x)&=\boldsymbol{H(\mathit{2},6)}\\
&=\uwave{\boldsymbol{H(3,\mathit{6})}}+\uwave{xq^2H(4,9)}\\
&=\uwave{\boldsymbol{H(\mathit{3},9)}}+\uwave{x^2q^6H(6,12)}+xq^2H(4,9)\\
&=\uwave{H(4,9)}+\uwave{xq^3H(5,12)}+x^2q^6H(6,12)+xq^2H(4,9)\\
&=F_1(xq^3)+xq^3F_6(xq^3)+x^2q^6F_7(xq^3)+xq^2F_5(xq^3).
\end{align*}
Finally,
\begin{align*}
F_7(x)&=\boldsymbol{H(3,\mathit{6})}\\
&=\uwave{\boldsymbol{H(\mathit{3},9)}}+\uwave{x^2q^6H(6,12)}\\
&=\uwave{H(4,9)}+\uwave{xq^3H(5,12)}+x^2q^6H(6,12)\\
&=F_1(xq^3)+xq^3F_6(xq^3)+x^2q^6F_7(xq^3).
\end{align*}
Identities \eqref{eq:ex-2-1}, \eqref{eq:ex-2-2} and \eqref{eq:ex-2-3} are therefore proved.

\subsection{Binary trees}\label{sec:b-trees}

Interestingly, the previous two proofs can be represented nicely by binary trees.

More precisely, all nodes are of the form $H(\beta_1,\ldots,\beta_r,\ldots,\beta_R)$. Then Theorem \ref{th:rec} gives two children of $H(\beta_1,\ldots,\beta_r,\ldots,\beta_R)$: the left child is $H(\beta_1,\ldots,\beta_r+A_r,\ldots,\beta_R)$, weighted by $1$, and the right child is $H(\beta_1+\alpha_{r,1},\ldots,\beta_r+\alpha_{r,r},\ldots,\beta_R+\alpha_{r,R})$, weighted by $x^{\gamma_r}q^{\beta_r}$. See Fig.~\ref{fig:node-children}.

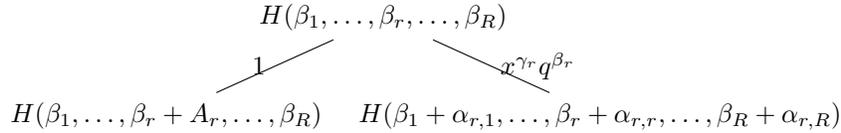
\begin{figure}[h]
	\caption{Node $H(\beta_1,\ldots,\beta_r,\ldots,\beta_R)$ and its children}\label{fig:node-children}
	\medskip
	\begin{forest}
		for tree={l sep=20pt}
		[{$H(\beta_1,\ldots,\beta_r,\ldots,\beta_R)$} 
			[{$H(\beta_1,\ldots,\beta_r+A_r,\ldots,\beta_R)$}, edge label={node[midway,left] {$1$}} ]
			[{$H(\beta_1+\alpha_{r,1},\ldots,\beta_r+\alpha_{r,r},\ldots,\beta_R+\alpha_{r,R})$}, edge label={node[midway,right] {$x^{\gamma_r}q^{\beta_r}$}} ] 
		]
	\end{forest}
\end{figure}

\medskip

Now the proofs of \eqref{eq:q-diff-sys-ex1} and \eqref{eq:q-diff-sys-ex2} can be illustrated by Figs.~\ref{fig:ex1} and \ref{fig:ex2}, respectively.

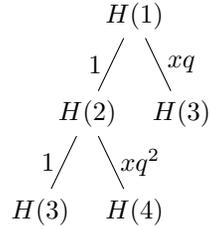
\begin{figure}[ht]
	\caption{The binary tree for \eqref{eq:q-diff-sys-ex1}}\label{fig:ex1}
	\medskip
	\begin{forest}
		for tree={l sep=20pt}
		[$H(1)$ 
			[$H(2)$, edge label={node[midway,left] {$1$}}  
				[$H(3)$, edge label={node[midway,left] {$1$}} ] 
				[$H(4)$, edge label={node[midway,right] {$xq^2$}} ] 
			]
			[$H(3)$, edge label={node[midway,right] {$xq$}} ]
		]
	\end{forest}
\end{figure}

\begin{figure}[ht]
	\caption{The binary tree for \eqref{eq:q-diff-sys-ex2}}\label{fig:ex2}
	\medskip
		\begin{forest}
		for tree={l sep=20pt}
		[{$H(1,\mathit{3})$} 
			[{$H(\mathit{1},6)$}, edge label={node[midway,left] {$1$}}  
				[{$H(\mathit{2},6)$}, edge label={node[midway,left] {$1$}} 
					[{$H(3,\mathit{6})$}, edge label={node[midway,left] {$1$}} 
						[{$H(\mathit{3},9)$}, edge label={node[midway,left] {$1$}} 
							[{$H(4,9)$}, edge label={node[midway,left] {$1$}} ]
							[{$H(5,12)$}, edge label={node[midway,right] {$xq^3$}} ]
						]
						[{$H(6,12)$}, edge label={node[midway,right] {$x^2q^6$}} ]
					]
					[{$H(4,9)$}, edge label={node[midway,right] {$xq^2$}} ]
				] 
				[{$H(\mathit{3},9)$}, edge label={node[midway,right] {$xq$}}
					[{$H(4,9)$}, edge label={node[midway,left] {$1$}} ]
					[{$H(5,12)$}, edge label={node[midway,right] {$xq^3$}} ]
				] 
			]
			[{$H(4,9)$}, edge label={node[midway,right] {$x^2q^3$}} ]
		]
	\end{forest}
\end{figure}
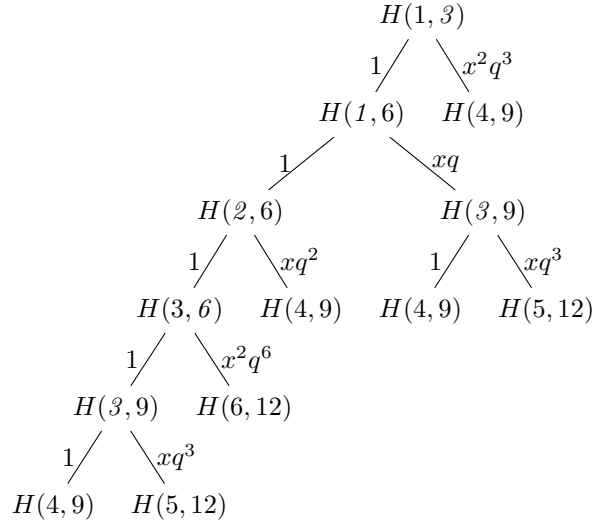

\medskip

In fact, it is relatively easy to deduce other much more complicated identities of the same flavor as \eqref{eq:q-diff-sys-ex1} and \eqref{eq:q-diff-sys-ex2}. For example, the next result follows from the binary tree in Fig.~\ref{fig:ex3}.

\begin{theorem}
	Let
	\begin{align}
	&F_1(x)=\cdots=F_6(x)\notag\\
	&\quad=\sum_{n_1,n_2,n_3\ge 0}\frac{q^{\frac{n_1^2}{2}+3n_2^2+\frac{9n_3^2}{2}+2 n_1 n_2+6n_2 n_3+3n_3n_1+\frac{n_1}{2}-n_2-\frac{n_3}{2}}x^{n_1+2n_2+3n_3}}{(q;q)_{n_1} (q^2;q^2)_{n_2}  (q^3;q^3)_{n_3}},\\
	&F_7(x)=\cdots=F_{13}(x)\notag\\
	&\quad=\sum_{n_1,n_2,n_3\ge 0}\frac{q^{\frac{n_1^2}{2}+3n_2^2+\frac{9n_3^2}{2}+2 n_1 n_2+6n_2 n_3+3n_3n_1+\frac{3n_1}{2}+n_2+\frac{5n_3}{2}}x^{n_1+2n_2+3n_3}}{(q;q)_{n_1} (q^2;q^2)_{n_2}  (q^3;q^3)_{n_3}},\\
	&F_{14}(x)=\cdots=F_{21}(x)\notag\\
	&\quad=\sum_{n_1,n_2,n_3\ge 0}\frac{q^{\frac{n_1^2}{2}+3n_2^2+\frac{9n_3^2}{2}+2 n_1 n_2+6n_2 n_3+3n_3n_1+\frac{3n_1}{2}+3n_2+\frac{11n_3}{2}}x^{n_1+2n_2+3n_3}}{(q;q)_{n_1} (q^2;q^2)_{n_2}  (q^3;q^3)_{n_3}}\\
	\intertext{and}
	&F_{22}(x)=F_{23}(x)\notag\\
	&\quad=\sum_{n_1,n_2,n_3\ge 0}\frac{q^{\frac{n_1^2}{2}+3n_2^2+\frac{9n_3^2}{2}+2 n_1 n_2+6n_2 n_3+3n_3n_1+\frac{5n_1}{2}+3n_2+\frac{11n_3}{2}}x^{n_1+2n_2+3n_3}}{(q;q)_{n_1} (q^2;q^2)_{n_2}  (q^3;q^3)_{n_3}}.
	\end{align}
	\addtocounter{MaxMatrixCols}{25}
	Let
	{
	$$\mA=\scalebox{0.925}{%
	$\begin{pmatrix}
	1&1&1&1&1&1&1&1&1&1&1&1&1&1&1&1&1&1&1&1&1&1&1\\
	1&1&1&1&1&1&1&1&1&1&1&1&1&1&1&1&1&1&1&1&1&1&1\\
	1&1&1&1&1&1&1&1&1&1&1&1&1&1&1&1&1&1&1&1&1&1&1\\
	1&1&1&1&1&1&1&1&1&1&1&1&1&1&1&1&1&1&1&1&1&1&1\\
	1&1&1&1&1&1&1&1&1&1&1&1&1&1&1&1&1&1&1&1&1&1&1\\
	1&1&1&1&1&1&1&1&1&1&1&1&1&1&1&1&1&1&1&1&1&1&1\\
	1&1&0&0&0&0&1&1&1&1&0&0&0&1&1&1&1&0&0&0&0&1&0\\
	1&1&0&0&0&0&1&1&1&1&0&0&0&1&1&1&1&0&0&0&0&1&0\\
	1&1&0&0&0&0&1&1&1&1&0&0&0&1&1&1&1&0&0&0&0&1&0\\
	1&1&0&0&0&0&1&1&1&1&0&0&0&1&1&1&1&0&0&0&0&1&0\\
	1&1&0&0&0&0&1&1&1&1&0&0&0&1&1&1&1&0&0&0&0&1&0\\
	1&1&0&0&0&0&1&1&1&1&0&0&0&1&1&1&1&0&0&0&0&1&0\\
	1&1&0&0&0&0&1&1&1&1&0&0&0&1&1&1&1&0&0&0&0&1&0\\
	1&1&0&0&0&0&1&1&0&0&0&0&0&1&1&1&0&0&0&0&0&1&0\\
	1&1&0&0&0&0&1&1&0&0&0&0&0&1&1&1&0&0&0&0&0&1&0\\
	1&1&0&0&0&0&1&1&0&0&0&0&0&1&1&1&0&0&0&0&0&1&0\\
	1&1&0&0&0&0&1&1&0&0&0&0&0&1&1&1&0&0&0&0&0&1&0\\
	1&1&0&0&0&0&1&1&0&0&0&0&0&1&1&1&0&0&0&0&0&1&0\\
	1&1&0&0&0&0&1&1&0&0&0&0&0&1&1&1&0&0&0&0&0&1&0\\
	1&1&0&0&0&0&1&1&0&0&0&0&0&1&1&1&0&0&0&0&0&1&0\\
	1&1&0&0&0&0&1&1&0&0&0&0&0&1&1&1&0&0&0&0&0&1&0\\
	1&0&0&0&0&0&1&0&0&0&0&0&0&1&1&0&0&0&0&0&0&1&0\\
	1&0&0&0&0&0&1&0&0&0&0&0&0&1&1&0&0&0&0&0&0&1&0
	\end{pmatrix}$
	}$$
	}
	and
	\begin{align*}
	\mW(x)=\diag(&1,xq^2,xq,x^2q^3,x^2q^2,x^3q^4,\\
	&xq^3,x^2q^5,x^2q^4,x^3q^7,x^2q^4,x^3q^6,x^3q^5,\\
	&x^2q^7,x^2q^6,x^3q^9,x^3q^8,x^3q^8,x^3q^7,x^4q^{10},x^4q^9,\\
	&x^3q^{10},x^4q^{11}).
	\end{align*}
	Then,
	\begin{equation}\label{eq:q-diff-sys-ex3-new}
	\scalebox{0.865}{%
		$\begin{pmatrix}
		F_1(x)\\
		F_2(x)\\
		\vdots\\
		F_{23}(x)
		\end{pmatrix}
		=\mA.\mW(x).\begin{pmatrix}
		F_1(xq^3)\\
		F_2(xq^3)\\
		\vdots\\
		F_{23}(xq^3)
		\end{pmatrix}$}.
	\end{equation}
\end{theorem}

\medskip

\begin{proof}
	Let $\uba=\begin{pmatrix}
	1 & 2 & 3\\
	2 & 6 & 6\\
	3 & 6 & 9
	\end{pmatrix}$, $\ubg=(1,2,3)$, $\ubA=(1,2,3)$ and $S=3$. We have
	\begin{gather*}
	F_1(x)=\cdots=F_6(x)=H(1,2,4)\quad\xrightarrow{x\mapsto xq^3}\quad H(4,8,13),\\
	F_7(x)=\cdots=F_{13}(x)=H(2,4,7)\quad\xrightarrow{x\mapsto xq^3}\quad H(5,10,16),\\
	F_{14}(x)=\cdots=F_{21}(x)=H(2,6,10)\quad\xrightarrow{x\mapsto xq^3}\quad H(5,12,19)\\
	\intertext{and}
	F_{22}(x)=F_{23}(x)=H(3,6,10)\quad\xrightarrow{x\mapsto xq^3}\quad H(6,12,19).
	\end{gather*}
	The rest follows from the binary tree in Fig.~\ref{fig:ex3}.
\end{proof}

\afterpage{%
	\begin{landscape}
		\noindent Recall that $\uba=\begin{pmatrix}
		1 & 2 & 3\\
		2 & 6 & 6\\
		3 & 6 & 9
		\end{pmatrix}$, $\ubg=(1,2,3)$, $\ubA=(1,2,3)$ and $S=3$.
		{\tiny
			\begin{figure}[b]
				\caption{The binary tree for \eqref{eq:q-diff-sys-ex3-new}}\label{fig:ex3}
				\medskip
				\begin{forest}
					for tree={l sep=15pt}
					[{$H(1,2,\textit{4})$}
						[{$H(1,\mathit{2},7)$}, edge label={node[midway,left] {$1$}}
							[{$H(\mathit{1},4,7)$}, edge label={node[midway,left] {$1$}}
								[{$H(2,4,\mathit{7})$}, edge label={node[midway,left] {$1$}}
									[{$H(2,\mathit{4},10)$}, edge label={node[midway,left] {$1$}}
										[{$H(\mathit{2},6,10)$}, edge label={node[midway,left] {$1$}}
											[{$H(3,6,\mathit{10})$}, edge label={node[midway,left] {$1$}}
												[{$H(3,\mathit{6},13)$}, edge label={node[midway,left] {$1$}}
													[{$H(\mathit{3},8,13)$}, edge label={node[midway,left] {$1$}}
														[{$H(4,8,13)$}, edge label={node[midway,left] {$1$}}]
														[{$H(\mathit{4},10,16)$}, edge label={node[midway,right] {$xq^{3}$}}
															[{$H(5,10,16)$}, edge label={node[midway,left] {$1$}}]
															[{$H(5,12,19)$}, edge label={node[midway,right] {$xq^{4}$}}]
														]
													]
													[{$H(5,12,19)$}, edge label={node[midway,right] {$x^2q^{6}$}}]
												]
												[{$H(6,12,19)$}, edge label={node[midway,right] {$x^3q^{10}$}}]
											]
											[{$H(\mathit{3},8,13)$}, edge label={node[midway,right] {$xq^2$}}
												[{$H(4,8,13)$}, edge label={node[midway,left] {$1$}}]
												[{$H(\mathit{4},10,16)$}, edge label={node[midway,right] {$xq^{3}$}}
													[{$H(5,10,16)$}, edge label={node[midway,left] {$1$}}]
													[{$H(5,12,19)$}, edge label={node[midway,right] {$xq^{4}$}}]
												]
											]
										]
										[{$H(\mathit{4},10,16)$}, edge label={node[midway,right] {$x^2q^4$}}
											[{$H(5,10,16)$}, edge label={node[midway,left] {$1$}}]
											[{$H(5,12,19)$}, edge label={node[midway,right] {$xq^{4}$}}]
										]
									]
									[{$H(5,10,16)$}, edge label={node[midway,right] {$x^3q^7$}}]
								]
								[{$H(\mathit{2},6,10)$}, edge label={node[midway,left] {$xq$}}
									[{$H(3,6,\mathit{10})$}, edge label={node[midway,left] {$1$}}
										[{$H(3,\mathit{6},13)$}, edge label={node[midway,left] {$1$}}
											[{$H(\mathit{3},8,13)$}, edge label={node[midway,left] {$1$}}
												[{$H(4,8,13)$}, edge label={node[midway,left] {$1$}}]
												[{$H(\mathit{4},10,16)$}, edge label={node[midway,right] {$xq^{3}$}}
													[{$H(5,10,16)$}, edge label={node[midway,left] {$1$}}]
													[{$H(5,12,19)$}, edge label={node[midway,right] {$xq^{4}$}}]
												]
											]
											[{$H(5,12,19)$}, edge label={node[midway,right] {$x^2q^{6}$}}]
										]
										[{$H(6,12,19)$}, edge label={node[midway,right] {$x^3q^{10}$}}]
									]
									[{$H(\mathit{3},8,13)$}, edge label={node[midway,right] {$xq^2$}}
										[{$H(4,8,13)$}, edge label={node[midway,left] {$1$}}]
										[{$H(\mathit{4},10,16)$}, edge label={node[midway,right] {$xq^{3}$}}
											[{$H(5,10,16)$}, edge label={node[midway,left] {$1$}}]
											[{$H(5,12,19)$}, edge label={node[midway,right] {$xq^{4}$}}]
										]
									]
								]
							]
							[{$H(\mathit{3},8,13)$}, edge label={node[midway,right] {$x^2q^2$}}
								[{$H(4,8,13)$}, edge label={node[midway,left] {$1$}}]
								[{$H(\mathit{4},10,16)$}, edge label={node[midway,right] {$xq^{3}$}}
									[{$H(5,10,16)$}, edge label={node[midway,left] {$1$}}]
									[{$H(5,12,19)$}, edge label={node[midway,right] {$xq^{4}$}}]
								]
							]
						]
						[{$H(4,8,13)$}, edge label={node[midway,right] {$x^3q^4$}}]
					]
				\end{forest}
			\end{figure}
		}
	\end{landscape}
}

\section{Closing remarks}\label{sec:rmks}

Our main concern is about the Factorization Property. Recall that $\mU$ is a $(0,1)$-matrix such that all entries in the first row and column are $1$, and $\mV$ is a diagonal matrix such that all (diagonal) entries are monic monomials in $x$ and $q$ with $\mV_{1,1}=1$. The Factorization Property says that
\begin{equation}\label{eq:f-p}
\ubF_{\ubb}(x) = \mU.\mV.\ubF_{\ubb}(xq^S),
\end{equation}
where $S$ is a positive integer and
$$\ubF_{\ubb}(x)=\begin{pmatrix}
F_1(x)\\
F_2(x)\\
\vdots\\
F_K(x)
\end{pmatrix}=
\begin{pmatrix}
H(\ubb_1)\\
H(\ubb_2)\\
\vdots\\
H(\ubb_K)
\end{pmatrix},$$
in which $H(\ubb)=H(\beta_1,\ldots,\beta_R)$ is of the form
\begin{align*}
H(\ubb)=\sum_{n_1,\ldots,n_R\ge 0}\frac{q^{\sum_{r=1}^R \alpha_{r,r}n_r(n_r-1)/2}q^{\sum_{1\le i< j\le R}\alpha_{i,j}n_i n_j}q^{\sum_{r=1}^R \beta_r n_r}x^{\sum_{r=1}^R \gamma_r n_r}}{(q^{A_1};q^{A_1})_{n_1}\cdots (q^{A_R};q^{A_R})_{n_R}}.
\end{align*}
Probably we also require the Additional Conditions: for all $1\le s\le R$:
\begin{enumerate}[label={\textup{(\roman*).}}]
	\item $\gamma_s S\in A_s\mathbb{Z}$;
	\item for all $1\le r\le R$, $\alpha_{r,s}\in A_s\mathbb{Z}$.
\end{enumerate}

\medskip

\begin{problem}
	For given $\mU$ and $\mV$, is it possible to determine if there exist $\ubF_{\ubb}(x)$ and $S$ such that \eqref{eq:f-p} is true?
\end{problem}

We have another problem from a different direction.

\begin{problem}
	Are there any criteria of $\ubF_{\ubb}(x)$ that we are always able to find $\mU$, $\mV$ and $S$ such that \eqref{eq:f-p} is true?
\end{problem}

The last problem is perhaps simpler.

\begin{problem}
	Can we construct a family of $\mU$, $\mV$, $\ubF_{\ubb}(x)$ and $S$ such that \eqref{eq:f-p} holds?
\end{problem}

If we are able to find such construction, then we may derive a family of span one linked partition ideals (or at least a family of modified directed graphs) with nice analytic generation functions.

\bibliographystyle{amsplain}

\end{document}